\documentclass[11pt,twoside,reqno]{amsart}

\usepackage{microtype}
\usepackage{cite}
\usepackage[OT1]{fontenc}
\usepackage{type1cm}
\usepackage{amssymb}
\usepackage{mathtools}
\usepackage{enumerate}
\usepackage{enumitem}
\usepackage{comment}
\usepackage{xcolor}
\usepackage{mathtools} 
\usepackage{mathtools,xparse}
\usepackage{indentfirst}
\usepackage{amsthm,amssymb,nccmath}
\usepackage{centernot}
\usepackage{float}
\usepackage{subfig}
\usepackage{tikz}
\usepackage{xcolor}
\usepackage{enumerate}

\usepackage{geometry}
\usepackage{hyperref}
\hypersetup{
  colorlinks=true,
  linkcolor=black,
  anchorcolor=black,
  citecolor=black,
  filecolor=black,      
  menucolor=red,
  runcolor=black,
  urlcolor=black,
}

\numberwithin{equation}{section}

\theoremstyle{plain}
\newtheorem{theorem}{Theorem}[section]

\newtheorem{proposition}[theorem]{Proposition}

\newtheorem{lemma}[theorem]{Lemma}

\theoremstyle{remark}

\theoremstyle{definition}

\newcommand{\R}{\mathbb{R}}

\newcommand{\N}{\mathbb{N}}

\newcommand{\xvec}{\underline{x}}
\newcommand{\yvec}{\underline{y}}
\newcommand{\vvec}{\underline{v}}
\newcommand{\zvec}{\underline{z}}
\newcommand{\tvec}{\underline{t}}
\newcommand{\wvec}{\underline{w}}

\newcommand{\ibar}{\bar{\imath}}
\newcommand{\jbar}{\bar{\jmath}}

\renewcommand{\geq}{\geqslant}
\renewcommand{\leq}{\leqslant}

\DeclareMathOperator{\proj}{proj}

\DeclarePairedDelimiter{\abs}{\lvert}{\rvert}
\DeclarePairedDelimiter{\norm}{\lVert}{\rVert}

\begin{document}

\title{On the dimension of planar self-affine sets with non-invertible maps} 

\author{Bal\'azs B\'ar\'any}
\address[Bal\'azs B\'ar\'any]
        {Department of Stochastics, Institute of Mathematics, Budapest University of Technology and Economics, M\H{u}egyetem rkp. 3., H-1111 Budapest, Hungary}
\email{balubsheep@gmail.com}

\author{Viktor K\"ortv\'elyesi}
\address[Viktor K\"ortv\'elyesi]
        {Department of Stochastics, Institute of Mathematics, Budapest University of Technology and Economics, M\H{u}egyetem rkp. 3., H-1111 Budapest, Hungary}
\email{vkortvelyesi@gmail.com }

\date{\today}
\thanks{The research was supported by the grants the grants NKFI FK134251, K142169, and the grant NKFI KKP144059 ”Fractal geometry and applications”.}

\begin{abstract}
  In this paper, we study the dimension of planar self-affine sets, of which generating iterated function system (IFS) contains non-invertible affine mappings. We show that under a certain separation condition the dimension equals to the affinity dimension for a typical choice of the linear-parts of the non-invertible mappings, furthermore, we show that the dimension is strictly smaller than the affinity dimension for certain choices of parameters.
\end{abstract}

\maketitle

\section{Introduction}
	
	Let $\mathcal{F}=\{f_i(\xvec)=A_i\xvec+\tvec_i\}_{i\in I}$ be a finite collection of affine self-maps of $\R^d$, where $A_i$ is a $d\times d$ matrix and $\tvec_i\in\R^d$, called iterated function system (IFS). Let us assume throughout the paper that $\|A_i\|<1$ for every $i\in I$, where $\|\cdot\|$ denotes the usual matrix norm induced by the Euclidean-norm on $\R^d$. The well-known theorem of Hutchinson \cite{Hutchinson1981} states that there exists a unique non-empty compact set $\Lambda$ such that it is invariant with respect to the IFS $\mathcal{F}$. That is,
	\begin{gather*}
		\Lambda = \bigcup\limits_{i\in I} f_i(\Lambda).
	\end{gather*}
	We call the set $\Lambda$ self-affine set or the attractor of the IFS $\mathcal{F}$. In the special case, when the affine mappings are similarity transformations, we call the set $\Lambda$ self-similar.
	
	The dimension theory of self-affine sets has been widely studied in the past decades. Throughout the paper, we will denote the Haussdorff dimension of a set $A\subset\R^d$ by $\dim_H(A)$ and the box-counting and the upper box-counting dimension by $\dim_B(A)$ and $\overline{\dim}_B(A)$ respectively. For the definition and basic properties, we direct the reader to Falconer \cite{Falconer1990}. 
	
	In the case, when the affine mappings are similarities and the IFS satisfies the strong separation condition, i.e. 
	$$
	f_i(\Lambda)\cap f_j(\Lambda)=\emptyset\text{ for every }i\neq j,
	$$
	Hutchinson \cite{Hutchinson1981} showed that the Hausdorff and box dimension of $\Lambda$ equals to the similarity dimension. Determining the dimension becomes significantly harder when overlaps occur between the cylinder sets $f_i(\Lambda)$, $i\in I$. This problem was studied in several papers, like B\'ar\'any \cite{barany}, Hochman \cite{Hochman2014, hochmanrd}, Simon and Solomyak \cite{simonsolomyak}, Solomyak \cite{Solomyak1998}, etc. for typical systems in some proper sense. In general, the similarity dimension serves always as an upper bound of the box dimension.
	
	Another challenging problem is, when there is a strict affinity between the maps of the IFS $\mathcal{F}$, that is, there exists a matrix with at least two different eigenvalues in modulus. Falconer \cite{Falconer1988} generalized the similarity dimension for that case, called the affinity dimension, and showed that if all the maps are invertible then it is an upper bound of the box dimension of the attractor. Furthermore, Falconer \cite{Falconer1988} showed, which was later extended by Solomyak~\cite{Solomyak1998}, that if $\|A_i\|<1/2$ then for Lebesgue typical translation parameters the affinity dimension equals to the Hausdorff dimension. 
	
	The dimension theory of self-affine sets has been widely studied in the recent years, see B\'ar\'any, Hochman and Rapaport \cite{BaranyHochmanRapaport2019}, B\'ar\'any, K\"aenm\"aki and Koivusalo \cite{BaranyKaenmakiKoivusalo2017}, Falconer and Kempton \cite{falconerkemption18}, Hochman and Rapaport \cite{HochmanRapaport2021}, Rapaport \cite{Rapaport18}. All of these papers were considering systems, where the affine mappings are invertible. The first steps in the direction of considering non-invertible mappings has been made recently by K\"aenm\"aki and Nissinen \cite{kaenmaki2022non}. They studied the relation between the dimension of the attractor and the dimension of the self-affine set formed by the invertible mappings of the IFS for typical and for separated systems.
	
	This paper is devoted to generalise the result of B\'ar\'any, Hochman and Rapaport \cite{BaranyHochmanRapaport2019} for planar self-affine sets which defining IFS contains non-invertible mappings. We will show that for typical choice of parameters the Hausdorff dimension equals to the affinity dimension, but there is a relatively large set of exceptions.  Before we state our main theorem in details, we need to introduce some notations and definitions. 
	
	Let us denote by $\mathbb{S}^{d-1}$ the unit sphere on $\R^d$, and let us denote by $$\mathbb{T}^d=\overbrace{\mathbb{S}^1\times\cdots\times\mathbb{S}^1}^{d\text{-times}}$$ the $d$-dimensional torus. Every planar, contracting matrix $A$ with $\mathrm{rank}(A)=1$ can be represented as $A=\rho\vvec\wvec^T$, where $\rho\in(0,1)$ and $\vvec,\wvec\in\mathbb{S}^1$. In particular, $\vvec$ is the unit vector generating the image space and $\wvec$ is the unit vector generating the kernel of the matrix $A$. Note that this representation is not unique.
	
	For a $2\times2$ matrix $A$, denote by $\alpha_1(A)\geq\alpha_2(A)\geq0$ the singular values. For every $t\geq0$, let $\varphi^t(A)$ be the singular value function defined as
	$$
	\varphi^t(A)=\begin{cases}
		\alpha_1(A)^t & \text{if }0\leq t\leq 1,\\
		\alpha_1(A)\alpha_2(A)^{t-1} & \text{if }1<t\leq2,\\
		(\alpha_1(A)\alpha_2(A))^{t/2} & \text{if }t>2.
	\end{cases}
	$$
	Note that if $A$ has rank one then $\varphi^t(A)=0$ for every $t>1$. 
	
	Let $I$ and $J$ be finite sets of indices and for every $i\in I\cup J$ let $f_i(\xvec)=A_i\xvec+t_i$ be an affine map such that $\|A_i\|<1$ for every $i\in I\cup J$,  $\mathrm{rank}(A_i)=2$ for every $i\in I$ and $\mathrm{rank}(A_i)=1$ for every $i\in J$.  Let us consider the following parametrized family of affine IFS 
	\begin{equation}\label{eq:IFSdef}
	\mathcal{F}_{\mathsf{w}}=\{f_i(\xvec)=A_i\xvec+\tvec_i\}_{i\in I}\bigcup\{f_i(\xvec)=\rho_i\vvec_i\wvec_i^T\xvec+\tvec_i\}_{i\in J},
	\end{equation}
	where $\mathsf{w}=(\wvec_j)_{j\in J}\in\left(\mathbb{S}^1\right)^{\#J}=\mathbb{T}^{\#J}$ is considered as the parameters while all the other quantities are fixed. We assume throughout the paper that $\#J\geq1$ and $\#(I\cup J)\geq2$ to avoid trivial cases. Let us denote the attractor of $\mathcal{F}_{\mathsf{w}}$ by $\Lambda_{\mathsf{w}}$. We define the affinity dimension $s(\mathcal{F}_{\mathsf{w}})$ of the self-affine IFS as
	\begin{gather}\label{CH3_def_aff_dim}
		s(\mathcal{F}_{\mathsf{w}}) =\min\left\{2, \inf \left\{ t \geq 0: \, \sum\limits_{n=1}^{\infty} \smashoperator[r]{\sum\limits_{i_1,\ldots,i_n\in I\cup J}} \varphi^t(A_{i_1}\cdots A_{i_n}) < \infty \right\}\right\}.
	\end{gather}
	The affinity dimension serves as a natural upper bound for the upper box-counting dimension, and in particular for the Hausdorff dimension, of the attractor in the non-invertible case too, see K\"aenm\"aki and Nissinen \cite{kaenmaki2022non}. 
	
	Let us define the affinity dimension of the sub-system formed by the invertible mappings $\mathcal{F}_{\mathrm{reg}}=\{f_i\}_{i\in I}$ as 
	$$
	s_{\mathrm{reg}}:=\min\left\{2,\inf\left\{s\geq0:\sum_{n=1}^\infty \smashoperator[r]{\sum_{i_1,\ldots,i_n\in I}}\varphi^s(A_{i_1}\cdots A_{i_n})<\infty\right\}\right\}.
	$$

	Let us observe that if $s(\mathcal{F}_{\mathsf{w}})>1$ then $s(\mathcal{F}_{\mathsf{w}})=s_{\mathrm{reg}}$. Furthermore, if $s_{\mathrm{reg}}\leq1$ then $s(\mathcal{F}_{\mathsf{w}})\leq1$ as well. The following was shown by K\"aenm\"aki and Nissinen \cite[Theorem~1.1(1)]{kaenmaki2022non}: suppose that $s_{\mathrm{reg}}\geq1$ and suppose that the IFS $\mathcal{F}_{\mathrm{reg}}=\{f_i\}_{i\in I}$ satisfies the strong separation condition and the matrices $\{A_i\}_{i\in I}$ do not preserve any finite collection of the proper subspaces of $\R^2$ then $\dim_H(\Lambda_{\mathsf{w}})=\dim_B(\Lambda_{\mathsf{w}})=s(\mathcal{F}_{\mathsf{w}})=s_{\mathrm{reg}}$ for every $\mathsf{w}\in\mathbb{T}^{\#J}$. Hence, in the remaining part of the paper we assume that $s_{\mathrm{reg}}<1$.
	
	We say that satisfies the convex separation condition uniformly, if there exists a convex compact set $U\subset\R^2$ such that
	\[
	\begin{split}
	&\bigcup_{i\in I\cup J} f_i(U)\subseteq U\text{ for every }\mathsf{w}\in\mathbb{T}^{\#J}\text{ and }\\
	&\left(\bigcup_{\mathsf{w}\in\mathbb{T}^{\#J}}f_i(U)\right)\bigcap \left(\bigcup_{\mathsf{w}\in\mathbb{T}^{\#J}}f_j(U)\right)=\emptyset\text{ for every }i\neq j.
	\end{split}
	\]
	In the second part of the assumption, the image $f_i(U)$ depends on $\mathsf{w}\in\mathbb{T}^{\#J}$ if and only if $i\in J$, and in this case only on the corresponding coordinate of $\mathsf{w}$. In particular, if $i\in J$ then $f_i(U)$ is a line-segment which is parallel to $\mathrm{Im}(A_i)$, and $\bigcup_{\mathsf{w}\in\mathbb{T}^{\#J}}f_i(U)$ is the smallest line segment parallel to $\mathrm{Im}(A_i)$ containing $f_i(U)$ for every $\underline{w}_i\in\mathbb{S}^1$. Note that because of non-invertibility, the convex separation condition does not imply that the second and higher iterates do not contain overlaps.
    For an example of such system, see Figure~\ref{fig:example}, which IFS consists of 3 invertible and 2 non-invertible mappings.

 Furthermore, we say that the IFS $\mathcal{F}_{\mathrm{reg}}=\{f_i(\xvec)=A_i\xvec+\tvec_i\}_{i\in I}$ is irreducible if there is no proper subspace $V$ of $\R^2$ such that $V$ is preserved by the all matrices $A_i$ for $i\in I$.
 
  \begin{figure}[H]
 	\begin{tikzpicture}[scale=0.4]
 		
 		\draw[rotate=30, line width= 0.6] (-1.3,0.3) ellipse (200pt and 200pt);
 		\draw coordinate (B1) at ( -1.6, 6.5);
 		\node at (B1) [font=\large, above = 1mm] {$U$};
 		
 		\draw coordinate (F1) at ( -6.8, 0.5);
 		\node at (F1) [font=\large, above = 1mm] {$f_1(U)$};
 		\draw coordinate (F2) at ( 2.5, 3);
 		\node at (F2) [font=\large, above = 1mm] {$f_2(U)$};
 		\draw coordinate (F3) at ( 2, -7);
 		\node at (F3) [font=\large, above = 1mm] {$f_3(U)$};
 		\draw coordinate (F4) at ( 3.5, -3);
 		\node at (F4) [red, font=\large, above = -3mm] {$\bigcup\limits_{\mathsf{w}}f_4(U)$};
 		\draw coordinate (F5) at ( -2, 3.5);
 		\node at (F5) [red, font=\large, above = 1mm] {$\bigcup\limits_{\mathsf{w}}f_5(U)$};
 		\draw coordinate (F6) at ( -4.7, 3.8);
 		
 		\draw[rotate=60] (-2.5,3.5) ellipse (110pt and 70pt);
 		\draw[rotate=10] (-1.3,-4) ellipse (70pt and 40pt);
 		\draw[rotate=30] (2.3,0) ellipse (80pt and 50pt);
 		\draw[rotate=40] (-2.5,3.5) ellipse (40pt and 30pt);
 		\draw[rotate=10] (-6,-1) ellipse (30pt and 20pt);
 		\draw[rotate=30] (-4,0) ellipse (20pt and 15pt);
 		\draw[rotate=40] (3.5,-1.3) ellipse (20pt and 10pt);
 		\draw[rotate=10] (2,1) ellipse (30pt and 20pt);
 		\draw[rotate=30] (0.5,0) ellipse (20pt and 15pt);
 		\draw[rotate=40] (-3,-3) ellipse (20pt and 10pt);
 		\draw[rotate=10] (-2,-3.5) ellipse (15pt and 10pt);
 		\draw[rotate=30] (-0.8,-4) ellipse (10pt and 5pt);
 		
 		
 		\draw[red,|-|,line width=0.3mm] (-6,4) -- (-2.2,3.7);
 		\draw[red,|-|,line width=0.3mm] (-1,-2) -- (3,-1);
 		
 		\draw[|-|,line width=0.15mm] (-5.4,4) -- (-4,3.9);
 		\draw[|-|,line width=0.15mm] (-4.2,3.9) -- (-3,3.8);
 		\draw[|-|,line width=0.15mm] (-3.2,3.8) -- (-2.6,3.7);
 		
 		\draw[|-|,line width=0.15mm] (-0.5,-1.9) -- (1,-1.5);
 		\draw[|-|,line width=0.15mm] (0.8,-1.6) -- (1.7,-1.3);
 		\draw[|-|,line width=0.15mm] (1.5,-1.4) -- (2.2,-1.2);
 		\draw[red,|-|,line width=0.15mm] (-6,-0.5) -- (-4,-1);
 		\draw[red,|-|,line width=0.15mm] (-3,-0.7) -- (-2.1,1.2);
 		\draw[red,|-|,line width=0.15mm] (1.5,0) -- (2.8,0.6);
 		\draw[red,|-|,line width=0.15mm] (1.5,2.4) -- (3.9,2.2);
 		\draw[red,|-|,line width=0.15mm] (0.3,-3.2) -- (1,-4.8);
 		\draw[red,|-|,line width=0.15mm] (-2.5,-4.4) -- (-0.5,-5.2);
 		
 	\end{tikzpicture}
 	\caption{The first and second level cylinder sets of an IFS\linebreak $\mathcal{F} = \{f_i(\xvec)=A_i\xvec+\tvec_i\}_{i\in \{1,2,3\}^*}\bigcup\{f_j(\xvec)=\rho_j\vvec_j\wvec_j^T\xvec+\tvec_j\}_{j\in \{4,5\}^*}$  satisfying the convex separation condition.}\label{fig:example}
 \end{figure}

	\begin{theorem}\label{thm:main}
	Let $\mathcal{F}_{\mathsf{w}}$ be a family of affine IFSs as in \eqref{eq:IFSdef} with attractor $\Lambda_{\mathsf{w}}$ such that $\mathcal{F}_{\mathsf{w}}$ contains at least two maps and contains at least one non-invertible affine map. Suppose that $s_{\mathrm{reg}}<1$, $\mathcal{F}_{\mathsf{w}}$ satisfies the convex separation condition uniformly for every $\mathsf{w}\in\mathbb{T}^{\#J}$ and $\mathcal{F}_{\mathrm{reg}}$ is irreducible. Then
	\begin{enumerate}
		\item\label{it:dimension} there exists a set $E_1\subseteq\mathbb{T}^{\#J}$ such that $\dim_HE_1\leq\#J-1$ and for every $\mathsf{w}\in\mathbb{T}^{\#J}\setminus E_1$,
		$$
		\dim_H(\Lambda_{\mathsf{w}})=\dim_B(\Lambda_{\mathsf{w}})=s(\mathcal{F}_{\mathsf{w}}).
		$$
		\item\label{it:exception} If $\sup_{\mathsf{w}}s(\mathcal{F}_{\mathsf{w}})<1$ then there exists a set $E_2\subseteq\mathbb{T}^{\#J}$ such that $\dim_HE_2=\#J-1$ and for every $\mathsf{w}\in E_2$
		$$
		\dim_H(\Lambda_{\mathsf{w}})\leq\overline{\dim}_B(\Lambda_{\mathsf{w}})<s(\mathcal{F}_{\mathsf{w}}).
		$$
	\end{enumerate}
	\end{theorem}
Let us note that it is not known whether the box-counting dimension of self-affine sets exists and equals to the Hausdorff dimension in general. Furthermore, if the matrices $\{A_i\}_{i\in I}$ do not preserve any finite collection of the proper subspaces of $\R^2$ then by B\'ar\'any, Hochman and Rapaport \cite{BaranyHochmanRapaport2019} the affinity dimension $s_{\mathrm{reg}}$ of the subsystem formed by the invertible mappings serves as a lower bound for $\dim_H(\Lambda_{\mathsf{w}})$ for every $\mathsf{w}\in\mathbb{T}^{\#J}$, however, it is not necessarily the case if $\{A_i\}_{i\in I}$ is only irreducible.

	
\section{Study of the affinity dimension}

Let us first introduce some notations used throughout the paper. For an index set $K$, let $K^*=\bigcup_{n=0}^\infty K^n$ be the set of every finite words formed by the symbols in $K$. For a finite word $\ibar \in K^*$, denote $\abs{\ibar}$ the length of $\ibar$. For $\ibar = i_1 \, \dots \, i_n\in K^*$, we denote by $A_{\ibar}$ the finite product $ A_{i_1} \cdot \, \dots \, \cdot A_{i_n}$, and by $f_{\ibar}=f_{i_1}\circ\cdots\circ f_{i_n}$ the finite composition. We also use the convention that $\emptyset\in K^*$ with $|\emptyset|=0$, moreover, $A_{\emptyset}$ and $f_{\emptyset}$ are the identity matrix and identity map of $\R^2$ respectively.

Let $A$ be a $2\times 2$ matrix and let $V$ be a proper subspace of $\R^2$. Then let us define the conditional norm of $A$ on $V$ by
\begin{gather*}
	\norm{A | V} = \sup_{x\in V} \frac{\norm{Ax}}{\norm{x}}.
\end{gather*}
For a $2\times 2$ matrix $A$ with $\mathrm{rank}(A)=1$, denote by $\mathrm{Im}(A)$ and by $\mathrm{Ker}(A)$ the image space and the kernel of $A$ respectively. Clearly, if $A=\rho\vvec\wvec^T$ then $\mathrm{Im}(A)=\langle\vvec\rangle$ and $\mathrm{Ker}(A)=\langle\zvec\rangle$, where $\wvec$ is perpendicular to $\zvec$ and $\langle\vvec\rangle$ denotes the subspace generated by $\vvec$. Let us also denote the usual Euclidean scalar product on $\R^2$ by $\langle\vvec,\wvec\rangle=\vvec^T\wvec$.

Clearly, for every $2\times 2$ matrices $A,B$ with $\mathrm{rank}(B)=1$, and for every subspace $V$ of $\R^2$, we get
\begin{equation}\label{eq:changeofnorm}
\|AB\|=\|A|\mathrm{Im}(B)\|\cdot\|B\|\text{ and }\|AB|V\|=\|A|\mathrm{Im}(B)\|\cdot\|B|V\|.
\end{equation}
Note that if $\mathrm{Ker}(B)=V$ then $\|B|V\|=0$.

\begin{lemma}\label{lem:norm}
	Let $\mathcal{F}_{\mathsf{w}}$ be a family of affine IFSs as in \eqref{eq:IFSdef}. Suppose that $\mathcal{F}_{\mathrm{reg}}=\{f_i(\xvec)=A_i\xvec+\tvec_i\}_{i\in I}$ is irreducible. Then there exist constants $C>0,K>0$ such that for every $\ibar\in I^*$ and for every $i,j\in J$, there exists $\jbar\in I^*$ with $|\jbar|\leq K$ such that
	$$
	\|A_iA_{\ibar}A_{\jbar}|\mathrm{Im}(A_j)\|\geq C\|A_{\ibar}\|.
	$$
	In particular,
	$$
	\smashoperator[r]{\sum_{\jbar\in\bigcup_{k=0}^KI^k}}\|A_{\ibar}A_{\jbar}|\mathrm{Im}(A_j)\|\geq \smashoperator[r]{\sum_{\jbar\in\bigcup_{k=0}^KI^k}}\|A_iA_{\ibar}A_{\jbar}|\mathrm{Im}(A_j)\|\geq C\|A_{\ibar}\|.
	$$
\end{lemma}

The proof is a slight modification of the proof of Feng \cite[Proposition~2.8]{Feng2009}.

\begin{proof}
	For every $i\in J$, $A_i=\rho_i\vvec_i\wvec_i^T$. Let us argue by contradiction. Suppose that for every $C,K>0$ there exist $\ibar\in I^*$ and $i,j\in J$ such that for every $\jbar\in I^*$ with $|\jbar|\leq K$ 
	$$
	\|A_iA_{\ibar}A_{\jbar}|\mathrm{Im}(A_j)\|=\rho_i\langle A_{\ibar}^T\wvec_i,A_{\jbar}\vvec_j\rangle< C\|A_{\ibar}\|.
	$$
	Letting $C\to0$ and $K\to\infty$, and taking an accumulation point $A'$ of $A_{\ibar}^T/\|A_{\ibar}\|$ we get that there exists $i',j'\in J$ such that
	$$
	\langle A'\wvec_{i'},A_{\jbar}\vvec_{j'}\rangle=0
	$$
	for every $\jbar\in I^*$. Hence, $V:=\left\langle\bigcup_{\jbar\in I^*}A_{\jbar}\vvec_{j'}\right\rangle=\langle A'\wvec_{i'}\rangle^{\perp}$ is a proper subspace of $\R^2$ invariant with respect to the matrices $\{A_i\}_{i\in I}$, which is a contradiction.
\end{proof}

For every $j\in J$, let us define 
$$
s_j(\mathsf{w}):=\inf\left\{s\geq0:\smashoperator[r]{\sum_{\ibar\in (I\cup J)^*}}\|A_jA_{\ibar}|\mathrm{Im}(A_j)\|^s<\infty\right\}.
$$
We note that a simple corollary of Feng and K\"aenm\"aki \cite[Proposition~1.2]{FengKaenmaki2011} is that
\begin{equation}\label{eq:sreglim}
	\lim_{s\searrow s_{\mathrm{reg}}}\sum_{\ibar\in I^*}\|A_{\ibar}\|^s=\infty.
\end{equation}

\begin{lemma}\label{lem:affdim}
	Let $\mathcal{F}_{\mathsf{w}}$ be a family of affine IFSs as in \eqref{eq:IFSdef}. Suppose that $s_{\mathrm{reg}}<1$ and $\mathcal{F}_{\mathrm{reg}}=\{f_i(\xvec)=A_i\xvec+\tvec_i\}_{i\in I}$ is irreducible. Then
	$$
	s(\mathcal{F}_{\mathsf{w}})=\min\{1,s_j(\mathsf{w})\}>s_{\mathrm{reg}}\text{ for every }j\in J,
	$$
	where $s(\mathcal{F}_{\mathsf{w}})$ is defined in \eqref{CH3_def_aff_dim}.
\end{lemma}

Let us note that the claim $s(\mathcal{F}_{\mathsf{w}})>s_{\mathrm{reg}}$ follows also by \cite[Lemma~2.9]{kaenmaki2022non}, however, for the sake of completeness, we give an alternative proof here.

\begin{proof}
	First, we show that $s_j(\mathsf{w})>s_{\mathrm{reg}}$ for every $j\in J$. Observe that by \eqref{eq:changeofnorm} and Lemma~\ref{lem:norm}
	\[
	\begin{split}
		\smashoperator[r]{\sum_{\ibar\in (I\cup J)^*}}\|A_jA_{\ibar}|\mathrm{Im}(A_j)\|^s&=\sum_{k=0}^\infty\sum_{j_1,\ldots,j_k\in J}\smashoperator[r]{\sum_{\ibar_0,\ldots,\ibar_{k}\in I^*}}\|A_jA_{\ibar_0}A_{j_1}A_{\ibar_1}\cdots A_{j_k}A_{\ibar_k}|\mathrm{Im}(A_j)\|^s\\
		&=\sum_{k=0}^\infty\sum_{\substack{j_1,\ldots,j_k\in J\\ j_0=j_{k+1}=j}}\sum_{\ibar_0,\ldots,\ibar_{k}\in I^*}\prod_{\ell=0}^k\|A_{j_\ell}A_{\ibar_\ell}|\mathrm{Im}(A_{j_{\ell+1}})\|^s\\
		&\geq\sum_{k=0}^\infty\sum_{\ibar_0,\ldots,\ibar_{k}\in \bigcup_{k=K}^\infty I^k}\prod_{\ell=0}^k\|A_{j}A_{\ibar_\ell}|\mathrm{Im}(A_{j})\|^s\\
		&\geq\sum_{k=0}^\infty\left(C\sum_{\ibar\in I^*}\|A_{\ibar}\|^s\right)^{k+1}.
	\end{split}
	\]
	By \eqref{eq:sreglim}, there exists $s>s_{\mathrm{reg}}$ such that $\sum_{\ibar\in I^*}\|A_{\ibar}\|^s>C^{-1}$ and so, $s_j(\mathsf{w})\geq s>s_{\mathrm{reg}}$.
	
	Now, let us show that $s_i(\mathsf{w})=s_j(\mathsf{w})$ for every $i,j\in J$. Again by \eqref{eq:changeofnorm} and Lemma~\ref{lem:norm}
		\[
	\begin{split}
		\smashoperator[r]{\sum_{\ibar\in (I\cup J)^*}}\|A_jA_{\ibar}|\mathrm{Im}(A_j)\|^s&\geq\smashoperator[l]{\sum_{\ibar_0,\ibar_{2}\in I^*}}\smashoperator[r]{\sum_{\ibar_1\in(I\cup J)^*}}\|A_jA_{\ibar_0}A_{i}A_{\ibar_1}A_iA_{\ibar_2}|\mathrm{Im}(A_j)\|^s\\
		&=\smashoperator[l]{\sum_{\ibar_0,\ibar_{2}\in I^*}}\smashoperator[r]{\sum_{\ibar_1\in(I\cup J)^*}}\|A_jA_{\ibar_0}|\mathrm{Im}(A_{i})\|^s\|A_iA_{\ibar_1}|\mathrm{Im}(A_i)\|^s\|A_iA_{\ibar_2}|\mathrm{Im}(A_j)\|^s\\
		&\geq C^2\smashoperator{\sum_{\ibar_1\in(I\cup J)^*}}\|A_iA_{\ibar_1}|\mathrm{Im}(A_i)\|^s.
	\end{split}
	\]
	Hence, the claim follows by symmetrical reasons.
	
	Finally, let us show that $s_j(\mathsf{w})=s(\mathcal{F}_{\mathsf{w}})$. Clearly, by \eqref{eq:changeofnorm}
	$$
	\smashoperator{\sum_{\ibar\in(I\cup J)^*}}\|A_{\ibar}\|^s\geq\smashoperator{\sum_{\ibar\in(I\cup J)^*}}\|A_jA_{\ibar}A_j\|^s=\|A_j\|^s\smashoperator{\sum_{\ibar\in(I\cup J)^*}}\|A_jA_{\ibar}|\mathrm{Im}(A_j)\|^s
	$$
	for every $j\in J$, and so $s_j(\mathsf{w})\leq s(\mathcal{F}_{\mathsf{w}})$.
	
	On the other hand, let us enumerate the elements of $J$ by $j_1,\ldots,j_{m}$. Then similarly to previous calculations we have,
	\begin{align*}
		\sum_{\ibar\in(I\cup J)^*} \|A_{\ibar}\|^s&= \smashoperator{\sum_{\ibar\in (I\cup J\setminus\{j_1\})^*}} \|A_{\ibar}\|^s + 
		\smashoperator[r]{\sum_{\ibar_1, \ibar_2\in(I\cup J\setminus\{j_1\})^*}} \|A_{\ibar_1} A_{j_1} A_{\ibar_2}\|^s 
		+\smashoperator[l]{\sum_{\ibar_1, \ibar_2\in(I\cup J\setminus\{j_1\})^*}} \smashoperator[r]{\sum_{\jbar \in (I\cup J)^*}} \|A_{\ibar_1} A_{j_1} A_{\jbar} A_{j_1} A_{\ibar_2}\|^s \\
		&\leq \smashoperator{\sum_{\ibar\in (I\cup J\setminus\{j_1\})^*}} \|A_{\ibar}\|^s + 
		\|A_{j_1}\|^s \left(\smashoperator[r]{\sum_{\ibar\in(I\cup J\setminus\{j_1\})^*}} \|A_{\ibar}\|^s\right)^2
		+\left(\smashoperator[r]{\sum_{\ibar\in(I\cup J\setminus\{j_1\})^*}} \|A_{\ibar}\|^s\right)^2\cdot\|A_{j_1}\|^{2s} \\
        & \quad \cdot \smashoperator{\sum_{\jbar \in (I\cup J)^*}} \|A_{j_1} A_{\jbar}|\mathrm{Im}(A_{j_1})\|^s.
\end{align*}
By induction, we get that for every $n\in[1,m-1]\cap\N$
\begin{align*}
	\smashoperator[r]{\sum_{\ibar\in(I\cup J\setminus\{j_k\}_{k=1}^n)^*}} \|A_{\ibar}\|^s&= \smashoperator[r]{\sum_{\ibar\in (I\cup J\setminus\{j_k\}_{k=1}^{n+1})^*}} \|A_{\ibar}\|^s + 
	\smashoperator[r]{\sum_{\ibar_1, \ibar_2\in(I\cup J\setminus\{j_k\}_{k=1}^{n+1})^*}} \|A_{\ibar_1} A_{j_{n+1}} A_{\ibar_2}\|^s \\
	& \quad +\smashoperator[l]{\sum_{\ibar_1, \ibar_2\in(I\cup J\setminus\{j_k\}_{k=1}^{n+1})^*}} \smashoperator[r]{\sum_{\jbar \in (I\cup J\setminus\{j_k\}_{k=1}^{n})^*}} \|A_{\ibar_1} A_{j_{n+1}} A_{\jbar} A_{j_{n+1}} A_{\ibar_2}\|^s \\
	&\leq \smashoperator{\sum_{\ibar\in (I\cup J\setminus\{j_k\}_{k=1}^{n+1})^*}} \|A_{\ibar}\|^s + 
	\|A_{j_{n+1}}\|^s \left(\smashoperator[r]{\sum_{\ibar\in(I\cup J\setminus\{j_k\}_{k=1}^{n+1})^*}} \|A_{\ibar}\|^s\right)^2
	+\left(\smashoperator[r]{\sum_{\ibar\in(I\cup J\setminus\{j_k\}_{k=1}^{n+1})^*}} \|A_{\ibar}\|^s\right)^2\cdot\|A_{j_{n+1}}\|^{2s} \\
    & \quad \cdot \smashoperator[r]{\sum_{\jbar \in (I\cup J)^*}} \|A_{j_{n+1}} A_{\jbar}|\mathrm{Im}(A_{j_{n+1}})\|^s.
\end{align*}
Thus, we get that $s(\mathcal{F}_{\mathsf{w}})\leq\max\{s_{\mathrm{reg}},\max_{j\in J}s_j(\mathsf{w})\}$, which implies the claim.

\end{proof}

Let us observe that for every $j\in J$, by Lemma~\ref{lem:norm}
\[
\begin{split}
\sum_{\ibar\in (I\cup J)^*}\|A_jA_{\ibar}|\mathrm{Im}(A_j)\|^s&=\sum_{k=0}^{\infty}\sum_{\ibar_0,\ldots,\ibar_k\in (I\cup J\setminus\{j\})^*}\|A_jA_{\ibar_0}A_j\cdots A_jA_{\ibar_k}|\mathrm{Im}(A_j)\|^s\\
&=\sum_{k=0}^{\infty}\left(\sum_{\ibar\in(I\cup J\setminus\{j\})^*}\|A_jA_{\ibar}|\mathrm{Im}(A_j)\|^s\right)^k.
\end{split}
\]
Hence,
$$
\smashoperator{\sum_{\ibar\in (I\cup J)^*}}\|A_jA_{\ibar}|\mathrm{Im}(A_j)\|^s<\infty\text{ if and only if }\smashoperator[r]{\sum_{\ibar\in(I\cup J\setminus\{j\})^*}}\|A_jA_{\ibar}|\mathrm{Im}(A_j)\|^s<1.
$$
In other words,
\begin{equation}\label{eq:sj}
s_j(\mathsf{w})=\inf\left\{s\geq0:\smashoperator[r]{\sum_{\ibar\in(I\cup J\setminus\{j\})^*}}\|A_jA_{\ibar}|\mathrm{Im}(A_j)\|^s<1\right\}.
\end{equation}

\begin{lemma}\label{lem:affdim2}
	Let $\mathcal{F}_{\mathsf{w}}$ be a family of affine IFSs as in \eqref{eq:IFSdef}. Suppose that $s_{\mathrm{reg}}<1$ and $\mathcal{F}_{\mathrm{reg}}=\{f_i(\xvec)=A_i\xvec+\tvec_i\}_{i\in I}$ is irreducible. Then for every $j\in J$
	$$
	\smashoperator[r]{\sum_{\ibar\in(I\cup J\setminus\{j\})^*}}\|A_jA_{\ibar}|\mathrm{Im}(A_j)\|^{s_j(\mathsf{w})}=1,
	$$
	and $s_j(\mathsf{w})$ is the unique solution of the equation above.
\end{lemma}

\begin{proof}
	For any subset $J'\subsetneq J$ and for any $j',j''\in J\setminus J'$, let us define the following maps
	$$
	H_{J',j',j''}(s):=\smashoperator[r]{\sum_{\ibar\in(I\cup J')^*}}\|A_{j'}A_{\ibar}|\mathrm{Im}(A_{j''})\|^{s}.
	$$
	By Lemma~\ref{lem:norm}, the map $H_{J',j',j''}$ is strictly monotone decreasing on its support. Furthermore, since $H_{J',j',j''}$ is the limit of an increasing sequence of continuous maps, we get that $H_{J',j',j''}$ is lower semi-continuous, i.e. $\liminf_{s\to s_0}H_{J',j',j''}(s)\geq H_{J',j',j''}(s_0)$ for every $s_0\in[0,\infty)$. In particular, for every $s_0\in\R$ such that $H_{J',j',j''}(s_0)<\infty$,
	\begin{equation}\label{eq:rightcont}
	\lim_{s\searrow s_0}H_{J',j',j''}(s)=H_{J',j',j''}(s_0).
	\end{equation}
	For any $j',j''\in J$,
	\[
	H_{\emptyset,j',j''}(s)=\sum_{\ibar\in I^*}\|A_{j'}A_{\ibar}|\mathrm{Im}(A_{j''})\|^{s}\leq\sum_{\ibar\in I^*}\|A_{\ibar}\|^{s},
	\]
	and so by \eqref{eq:sreglim},  $H_{\emptyset,j',j''}(s)<\infty$ for every $s>s_{\mathrm{reg}}$. On the other hand, by Lemma~\ref{lem:norm}
	\[
	H_{\emptyset,j',j''}(s)=\sum_{\ibar\in I^*}\|A_{j'}A_{\ibar}|\mathrm{Im}(A_{j''})\|^{s}\geq\sum_{\substack{\ibar\in I^*\\|\ibar|\geq K}}\|A_{\ibar}\|^{s}\geq C\sum_{\ibar\in I^*}\|A_{\ibar}\|^s.
	\]
	Hence, by \eqref{eq:sreglim}, $\lim_{s\searrow s_{\mathrm{reg}}}H_{\emptyset,j',j''}(s)=\infty$.
	
	Let us recall the basic facts that for every $x\in\R$, $e^x\geq1+x$ and for every $\varepsilon>0$ there exists a $c=c(\varepsilon)>0$ such that $x^\varepsilon\log x\leq c(\varepsilon)$ for every $x\in[0,1]$. Thus, for every $s_1>s_2>s_{\mathrm{reg}}$
	\[
	\begin{split}
	0<H_{\emptyset,j',j''}(s_2)-H_{\emptyset,j',j''}(s_1)&=\smashoperator[r]{\sum_{\substack{\ibar\in I^*\\ \mathrm{Ker}(A_{j'}A_{\ibar})\neq\mathrm{Im}(A_{j''})}}}\|A_{j'}A_{\ibar}|\mathrm{Im}(A_{j''})\|^{s_2}(1-\|A_{j'}A_{\ibar}|\mathrm{Im}(A_{j''})\|^{s_1-s_2})\\
	&\leq (s_2-s_1)\smashoperator{\sum_{\substack{\ibar\in I^*\\ \mathrm{Ker}(A_{j'}A_{\ibar})\neq\mathrm{Im}(A_{j''})}}}\|A_{j'}A_{\ibar}|\mathrm{Im}(A_{j''})\|^{s_2}\log\|A_{j'}A_{\ibar}|\mathrm{Im}(A_{j''})\|\\
	&\leq (s_1-s_2)cH_{\emptyset,j',j''}(s_2-\varepsilon),
	\end{split}
	\]
	where $s_2-s_{\mathrm{reg}}>\varepsilon$. Thus the map $H_{\emptyset,j',j''}\colon(s_{\mathrm{reg}},\infty)\mapsto\R_+$ is strictly monotone decreasing and continuous, and so, there exists a unique $d_{\emptyset,j',j''}\in(s_{\mathrm{reg}},\infty)$ such that $H_{\emptyset,j',j''}(d_{\emptyset,j',j''})=1$.
	
	Let $j\in J$ be arbitrary but fixed. Let us enumerate the elements of $J$ by $j_1,\ldots,j_m$ such that $j_m=j$. Let $J_{k}:=\{j_1,\ldots,j_k\}$. Let us argue by induction. Namely, suppose that $d_{J_{k},j',j''}$ is well defined for every $j',j''\notin J_k$ and  $H_{J_k,j',j''}(s)\colon(d_{J_{k-1},j_{k},j_{k}},\infty)\mapsto\R_+$ is strictly monotone decreasing and continuous, $\lim_{s\searrow d_{J_{k-1},j_{k},j_{k}}}H_{J_{k},j',j''}(s)=\infty$ and $H_{J_k,j',j''}(d_{J_k,j',j''})=1$ for a unique $d_{J_k,j',j''}>d_{J_{k-1},j_k,j_k}$.
	
	Let $j',j''\notin J_{k+1}$ be arbitrary but fixed. By \eqref{eq:changeofnorm}
	\[
	\begin{split}
	H_{J_{k+1},j',j''}(s)&=H_{J_k,j',j''}(s)+H_{J_k,j',j_{k+1}}(s)H_{J_k,j_{k+1},j''}(s)\\
	&\quad+H_{J_k,j',j_{k+1}}(s)H_{J_k,j_{k+1},j''}(s)\sum_{n=1}^\infty H_{J_k,j_{k+1},j_{k+1}}(s)^n.
	\end{split}
	\]
	Since $H_{J_k,j',j''}(s),H_{J_k,j',j_{k+1}}(s),H_{J_k,j_{k+1},j''}(s)<\infty$ for every $s>d_{J_{k-1},j_k,j_k}$ we get that 
	$$
	H_{J_{k+1},j',j''}(s)<\infty\text{ if and only if }s>d_{J_k,j_{k+1},j_{k+1}},
	$$
	and
	$$
	\smashoperator[r]{\lim_{s\searrow d_{J_k,j_{k+1},j_{k+1}}}}H_{J_{k+1},j',j''}(s)=\infty.
	$$
	It is enough then to show that $H_{J_{k+1},j',j''}(s)$ is continuous, since the strict monotonicity follows by Lemma~\ref{lem:norm}. But similarly to the case $H_{\emptyset,j',j''}$, we get that for any $s_1>s_2>d_{J_k,j_{k+1},j_{k+1}}$
	\[
	\begin{split}
		0<&H_{J_{k+1},j',j''}(s_2)-H_{J_{k+1},j',j''}(s_1)\\
		&=\smashoperator[r]{\sum_{\substack{\ibar\in (I\cup J_{k+1})^*\\ \mathrm{Ker}(A_{j'}A_{\ibar})\not\supset\mathrm{Im}(A_{j''})}}}\|A_{j'}A_{\ibar}|\mathrm{Im}(A_{j''})\|^{s_2}(1-\|A_{j'}A_{\ibar}|\mathrm{Im}(A_{j''})\|^{s_1-s_2})\\
		&\leq (s_2-s_1)\smashoperator{\sum_{\substack{\ibar\in(I\cup J_{k+1})^*\\ \mathrm{Ker}(A_{j'}A_{\ibar})\not\supset\mathrm{Im}(A_{j''})}}}\|A_{j'}A_{\ibar}|\mathrm{Im}(A_{j''})\|^{s_2}\log\|A_{j'}A_{\ibar}|\mathrm{Im}(A_{j''})\|\\
		&\leq (s_1-s_2)cH_{J_{k+1},j',j''}(s_2-\varepsilon),
	\end{split}
	\]
	where $s_2-d_{J_k,j_{k+1},j_{k+1}}>\varepsilon$. Thus, the map $H_{J_{k+1},j',j''}\colon(d_{J_k,j_{k+1},j_{k+1}},\infty)\mapsto\R_+$ is strictly monotone decreasing and continuous. In particular, there exists a unique $d_{J_{k+1},j_{k+2},j_{k+2}}>d_{J_k,j_{k+1},j_{k+1}}$ such that $H_{J_{k+1},j_{k+2},j_{k+2}}(d_{J_{k+1},j_{k+2},j_{k+2}})=1$.
\end{proof}

\section{Lower bound of the dimension}

Let us introduce the natural mapping $\wvec\colon\R\mapsto\mathbb{S}^1$ as
$$
\wvec(\alpha)=\begin{pmatrix}
	\cos(\alpha) \\ \sin(\alpha)
\end{pmatrix}.
$$
This section is devoted to show the following proposition.

\begin{proposition}\label{CH3_prop_dimNew}
	Let $I$ and $J$ be finite collections of indices such that $J$ is non-empty. Let $\mathcal{F}_{\mathrm{reg}}=\{f_i(\xvec)=A_i\xvec+\tvec_i\}_{i\in I}$ be an irreducible IFS of invertible affine mappings such that $s_{\mathrm{reg}}<1$. Furthermore, let $\rho_j\in(0,1)$, $\vvec_j\in\mathbb{S}^1$, $c_j\in\R$, $\tvec_j\in\R^2$ and $\beta_j\in\R_+$ be arbitrary but fixed for every $j\in J$. Finally, for $\alpha\in\R$ let
	\begin{equation}\label{eq:specsystem}
	\mathcal{F}_{\alpha}=\mathcal{F}_{\mathrm{reg}}\cup\{f_j(\xvec)=\rho_j\vvec_j\wvec(c_j+\beta_j\alpha)^T\xvec+\tvec_j\}_{j\in J}
	\end{equation}
	be an IFS of affine mappings. Suppose that $\mathcal{F}_{\alpha}$ satisfies the convex separation condition uniformly for $\alpha\in\R$. Then there exists a set $E \subset \R$ such that $\dim_H(E) = 0$ and for every $\alpha \in \R \setminus E$ 
	\begin{gather*}
		\dim_H(\Lambda_{\alpha}) = s(\mathcal{F}_\alpha), 
	\end{gather*}
	where $\Lambda_{\alpha}$ is the attractor of the IFS $\mathcal{F}_{\alpha}$.
\end{proposition}

First, let us show why Theorem~\ref{thm:main}\eqref{it:dimension} follows from Proposition~\ref{CH3_prop_dimNew}.

\begin{proof}[Proof of Theorem~\ref{thm:main}\eqref{it:dimension}]
	Let $\mathcal{F}_{\mathsf{w}}$ be an IFS satisfying the assumptions of Theorem~\ref{thm:main}. By K\"aenm\"aki and Nissinen \cite[Lemma~3.2]{kaenmaki2022non}, we have that
	$$
	\overline{\dim}_B(\Lambda)\leq s(\mathcal{F}_{\mathsf{w}})\text{ for every }\mathsf{w}\in\mathbb{T}^{\#J}.
	$$
	Thus, it is enough to verify the lower bound.
	
	Let us argue by contradiction. That is, suppose that there exists a set $E\subset\mathbb{T}^{\#J}$ with $\dim_HE>\#J-1$ such that 
	$$
	\dim_H(\Lambda_{\mathsf{w}})<s(\mathcal{F}_{\mathsf{w}})\text{ for every }\mathsf{w}\in E.
	$$
	By using the map $\mathsf{w}\colon[0,2\pi]^{\#J}\mapsto\mathbb{T}^{\#J}$ defined as
	$$
	\mathsf{w}(\alpha_1,\ldots,\alpha_{\#J})=(\wvec(\alpha_1),\ldots,\wvec(\alpha_{\#J})),
	$$
	we get that there exists a set $E'\subset[0,2\pi]^{\#J}$ with $\dim_HE'>\#J-1$ such that
	$$
	\dim_H(\Lambda_{\mathsf{w}(\underline{\alpha})})<s(\mathcal{F}_{\mathsf{w}(\underline{\alpha})})\text{ for every }\underline{\alpha}=(\alpha_1,\ldots,\alpha_{\#J})\in E'.
	$$
	
	For a $\zvec\in\mathbb{S}^{\#J-1}$, let us denote the orthogonal projection from $\R^{\#J}$ to the $\#J-1$-dimensional subspace $\langle\zvec\rangle^\perp$ by $\mathrm{proj}_{\zvec}$. Let us denote the spherical measure on $\mathbb{S}^{\#J-1}$ by $\sigma_{\#J-1}$. By Mattila's slicing theorem, see \cite[Theorem~10.10]{Mattila1995}
	\[
	\mathcal{L}_{\#J-1}(\{\yvec\in\mathrm{proj}_{\zvec}(E'):\dim_H(E'\cap\proj_{\zvec}^{-1}(\yvec))>0\})>0\text{ for $\sigma_{\#J-1}$-a.e. }\zvec.
	\]
	Let $\zvec\in\mathbb{S}^{\#J-1}$ and $\yvec\in\R^{\#J}$ be such that $\dim_H(E'\cap\proj_{\zvec}^{-1}(\yvec))>0$. Hence, there exists a set $E''\subset\R$ with $\dim_H(E'')>0$ such that for every $\alpha\in E''$
	$$
	\dim_H(\Lambda_{\mathsf{w}(\yvec+\alpha\zvec)})<s(\mathcal{F}_{\mathsf{w}(\yvec+\alpha\zvec)}),
	$$
	which contradicts to Proposition~\ref{CH3_prop_dimNew}.
\end{proof}

\subsection{Hochman's theorem} Let us recall a theorem of Hochman \cite[Corollary~1.2]{Hochman2014}, which will be used to prove Proposition~\ref{CH3_prop_dimNew}. Let $\mathcal{I} \subset \R$ be a compact interval, and let $K$ be a finite set of indices. For every $k\in K$, let $\lambda_k \colon \mathcal{I} \to (-1,1) \setminus \{0\}$ and $a_k \colon \mathcal{I} \to \R$ be real analytic mappings, and let $\Phi_\alpha=\{g_k^{(\alpha)}(x) = \lambda_k(\alpha)x + a_k(\alpha)\}_{k\in K}$ be a parametrized family of IFS of contracting similarities on the real line with parameters $\alpha\in\mathcal{I}$. For every $\ibar=(i_1,i_2,\ldots),\jbar=(j_1,j_2,\ldots)\in K^\N$ let
$$
\Delta_{\ibar, \jbar}(\alpha) = \sum_{k=1}^\infty t_{i_k}(\alpha)\prod_{\ell=1}^{k-1}\lambda_{i_\ell}(\alpha)-\sum_{k=1}^\infty t_{j_k}(\alpha)\prod_{\ell=1}^{k-1}\lambda_{j_\ell}(\alpha).
$$

\begin{theorem}[Hochman]\label{CH2_thm_Hoch2}
	Let $\mathcal{I} \subset \R$ be a compact interval, and let $K$ be a finite set of indices. For every $k\in K$, let $\lambda_k \colon \mathcal{I} \to (-1,1) \setminus \{0\}$ and $a_k \colon \mathcal{I} \to \R$ be real analytic mappings, and let $\Phi_\alpha=\{g_k^{(\alpha)}(x) = \lambda_k(\alpha)x + a_k(\alpha)\}_{k\in K}$ be a parametrized family of IFS of contracting similarities on the real line with parameters $\alpha\in\mathcal{I}$. Denote the attractor of $\Phi_\alpha$ by $\Gamma_\alpha$. Suppose that 
	\begin{gather*}
		\Delta_{\ibar,\jbar} \equiv 0 \text{ on } \mathcal{I} \text{ if and only if } \ibar = \jbar.
	\end{gather*}
	Then there exists a set $E$ with $\dim_PE=0$ such that for every $\alpha\in\mathcal{I}\setminus E$
	$$
	\dim_H(\Gamma_\alpha)=\dim_B(\Gamma_\alpha)=\min\{1,s_0(\alpha)\},\text{ where }\sum_{k\in K}|\lambda_k(\alpha)|^{s_0(\alpha)}=1.
	$$
\end{theorem}

\subsection{Verifying Hochman's condition}

Recall that for a $\zvec\in\mathbb{S}^{1}$, we denote the orthogonal projection from $\R^{2}$ to the $1$-dimensional subspace $\langle\zvec\rangle^\perp$ by $\mathrm{proj}_{\zvec}$. The following is our main geometric lemma.

\begin{lemma}
	\label{CH8_convex}
	Let $A$ and $B$ convex compact sets, such that $A \cap B  = \emptyset$. If the Convex Separation Condition holds,
	then there exists an open set $O \subset \mathbb{S}^1$, such that 
	\begin{gather*}
		\mathrm{proj}_{\zvec}(A) \cap \mathrm{proj}_{\zvec}(B) = \emptyset\text{ for every }\zvec\in O.
	\end{gather*}
\end{lemma}

\begin{proof}
	Let $\vec{AB} = \{ \underline{a} - \underline{b} :  \underline{a} \in A \text{ and } \underline{b} \in B \}$ be the set of every vectors directing from $A$ to $B$. Then let us define cone 
	$$
	\mathcal{C} = \left\{ \frac{\underline{v}}{\norm{\underline{v}}} \text{: } \underline{v} \in \vec{AB} \right\}\text{ and }-\mathcal{C}=\left\{ -\frac{\underline{v}}{\norm{\underline{v}}} \text{: } \underline{v} \in \vec{AB} \right\}.
	$$
	Then $\mathcal{C}$ and $-\mathcal{C}$ are closed and compact. It is enough to show that the open set $\mathbb{S}^1\setminus(\mathcal{C}\cup-\mathcal{C})$ is non-empty. So $\mathrm{proj}_{\wvec}(A) \cap \mathrm{proj}_{\wvec}(B) = \emptyset$ for every $\wvec\in O$.
	
	If there is no such $\underline{w}$, i.e. $\mathcal{C}\cup-\mathcal{C}=\mathbb{S}^1$, then $\mathcal{C}\cap-\mathcal{C}\neq\emptyset$. Let $\wvec\in\mathcal{C}\cap-\mathcal{C}$. Let $\underline{a}_1,\underline{a}_2\in A$ and $\underline{b}_1,\underline{b}_2\in B$ be point such that,
	\begin{gather*}
		\underline{w} = \frac{\underline{a}_1-\underline{b}_1}{\norm{\underline{a}_1-\underline{b}_1}}= \frac{\underline{b}_2-\underline{a}_2}{\norm{\underline{b}_2-\underline{a}_2}}.
	\end{gather*}
	By connecting the endpoints  we might define a (possibly degenerate) trapeze.
	
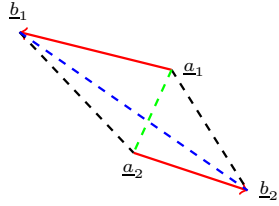
\begin{figure}[H]
		\centering
		\begin{tikzpicture}
			\draw coordinate (A1) at (-0.5,1.5);
			\node at (A1) [font=\tiny, right = 0.2mm] {$\underline{a}_1$};
			\draw coordinate (A2) at (-2.5,2);
			\node at (A2) [font=\tiny, above = 0.2mm] {$\underline{b}_1$};
			\draw coordinate (B1) at (-1,0.4);
			\node at (B1) [font=\tiny, below = 0.2mm] {$\underline{a}_2$};
			\draw coordinate (B2) at (0.5,-0.1);
			\node at (B2) [font=\tiny, right = 0.2mm] {$\underline{b}_2$};
			
			\draw[->, red, line width=0.3mm] (-0.5,1.5) -- (-2.5,2);
			\draw[->, red, line width=0.3mm] (-1,0.4) -- (0.5,-0.1);
			\draw[dashed, line width=0.3mm] (-0.5,1.5) -- (0.5,-0.1);
			\draw[dashed, line width=0.3mm] (-2.5,2) -- (-1,0.4);
			\draw[dashed, green, line width=0.3mm] (-0.5,1.5) -- (-1,0.4);
			\draw[dashed, blue, line width=0.3mm] (-2.5,2) -- (0.5,-0.1);
		\end{tikzpicture}
	\end{figure}
	Let $P_b =\{t\underline{b}_1+(1-t)\underline{b}_2:t\in[0,1]\}$ be the line between $\underline{b}_1$ and $\underline{b}_2$ and let 
	$P_a = \{t\underline{a}_1+(1-t)\underline{a}_2:t\in[0,1]\}$  be the line between $\underline{a}_1$ and $\underline{a}_2$. The diagonals of a (possibly degenerate) trapeze always intersect each other and so $P_a\cap P_b\neq\emptyset$. However,	since $A$ and $B$ are convex closed sets then, $P_b(x,y) \subset B$ and $P_a(x,y) \subset A$ which implies that, these lines cannot intersect each other, which is a contradiction.
\end{proof}

Let $\mathcal{F}_\alpha$ be the IFS defined in \eqref{eq:specsystem} satisfying the conditions of Proposition~\ref{CH3_prop_dimNew}. For simplicity, let us denote by $A_j^{(\alpha)}$ the matrices $\rho_j\vvec_j\wvec(c_j+\beta_j\alpha)^T$ for $j\in J$, and the products by $A_{\ibar}^{(\alpha)}$ for $\ibar\in(I\cup J)^*$ to emphasize its possible dependence on $\alpha$.

Let us define the natural projection $\Pi_\alpha$ from the symbolic space $\Sigma=(I\cup J)^\N$ to the attractor $\Lambda_\alpha$ of $\mathcal{F}_\alpha$ by
$$
\Pi_\alpha(\ibar)=\lim_{n\to\infty}f_{i_1}\circ\cdots\circ f_{i_n}(\underline{0}),
$$
for $\ibar=(i_1,i_2,\ldots)\in\Sigma$. For any $\ibar,\jbar\in\Sigma$, let $|\ibar\wedge\jbar|=\min\{k\geq1:i_k\neq j_k\}-1$, and let $\ibar\wedge\jbar=(i_1,\ldots,i_{|\ibar\wedge\jbar|})$. Denote the left-shift operator on $\Sigma$ by $\sigma$. Then clearly
$$
\Pi_\alpha(\ibar)=f_{i_1}(\Pi_{\alpha}(\sigma\ibar)).
$$

\begin{lemma}
	\label{lemma_for_Hochman}
Let $\mathcal{F}_\alpha$ be the IFS defined in \eqref{eq:specsystem} satisfying the conditions of Proposition~\ref{CH3_prop_dimNew}. Then
	\begin{gather*}
		\Pi_{\alpha}(\ibar) \equiv \Pi_{\alpha}(\jbar) \text{ for every } \alpha\in\R \text{ if and only if } \ibar=\jbar \in \Sigma.
	\end{gather*}
\end{lemma}

\begin{proof}
	Let $U\subset\R^2$ be the compact convex set with respect to the uniform convex separation condition holds. Let $\ibar,\jbar\in\Sigma$ be such that $\ibar\neq\jbar$. Then by using the linearity of the maps of $\mathcal{F}_\alpha$,
	$$
	\Pi_{\alpha}(\ibar)-\Pi_{\alpha}(\jbar)=A_{\ibar\wedge\jbar}^{(\alpha)}\left(\Pi_{\alpha}(\sigma^{|\ibar\wedge\jbar|}\ibar)-\Pi_{\alpha}(\sigma^{|\ibar\wedge\jbar|}\jbar)\right).
	$$
	By the uniform convex separation condition we get that $\Pi_{\alpha}(\sigma^{|\ibar\wedge\jbar|}\ibar)-\Pi_{\alpha}(\sigma^{|\ibar\wedge\jbar|}\jbar)\neq\underline{0}$ for every $\alpha\in\R$. Hence, 
	$$
	\Pi_{\alpha}(\ibar) \equiv \Pi_{\alpha}(\jbar)\text{ if and only if }\Pi_{\alpha}(\sigma^{|\ibar\wedge\jbar|}\ibar)-\Pi_{\alpha}(\sigma^{|\ibar\wedge\jbar|}\jbar)\in\mathrm{Ker}(A_{\ibar\wedge\jbar}^{(\alpha)})\text{ for all }\alpha\in\R.
	$$
	For a fixed $\alpha\in\R$, there are three possibilities: 

    \begin{enumerate}[label = (\Roman*.)]
        \item $\mathrm{rank}(A_{\ibar\wedge\jbar}^{(\alpha)})=2$ (i.e. $\ibar\wedge\jbar\in I^*$) but then $\mathrm{Ker}(A_{\ibar\wedge\jbar}^{(\alpha)})=\{\underline{0}\}$, which cannot happen;
        \item $\mathrm{rank}(A_{\ibar\wedge\jbar}^{(\alpha)})=1$ then $\mathrm{Ker}(A_{\ibar\wedge\jbar}^{(\alpha)})=\mathrm{Ker}(A_{i_k}^{(\alpha)}A_{\ibar'})$, where $k=\max\{1\leq\ell\leq|\ibar\wedge\jbar|:i_{\ell}\in J\}$ and $\ibar'\in I^{|\ibar\wedge\jbar|-k-1}$ is the suffix of $\ibar\wedge\jbar$;
        \item $\mathrm{Ker}(A_{\ibar\wedge\jbar}^{(\alpha)})=\R$, namely, there exists $1\leq k<\ell\leq|\ibar\wedge\jbar|$ and $\ibar'\in I^*$ such that
	    $\mathrm{Im}(A_{\ibar'}^{(\alpha)}A_{i_{\ell}})=\mathrm{Ker}(A_{i_k}^{(\alpha)})$ and $i_\ell,i_k\in J$.
    \end{enumerate}
	
	By the definition of $A_j^{(\alpha)}=\rho_j\vvec_j\wvec(c_j+\beta_j\alpha)^T$ for $j\in J$, $\mathrm{Im}(A_{\ibar'}A_{j}^{(\alpha)})$ is independent of $\alpha$ for any $\ibar'\in I^*$ and $j\in J$, it is clear that the set $\{\beta\in\R:\mathrm{Ker}(A_{\ibar\wedge\jbar}^{(\alpha)})=\R\}$ is a discrete and countable set for every $\ibar\neq\jbar\in\Sigma$. Hence, it is enough to check that for every $\ibar\in I^*$, $j\in J$ and $k_1\neq k_2\in I\cup J$ there exists $\alpha\in\R$ such that  
	\begin{equation}\label{eq:enough}
	\proj_{A_{\ibar}^T\underline{w}(c_j+\alpha\beta_j)}\left(f_{k_1}(U)\right)\cap\proj_{A_{\ibar}^T\underline{w}(c_j+\alpha\beta_j)}\left(f_{k_2}(U)\right)=\emptyset.
	\end{equation}
	But by Lemma~\ref{CH8_convex}, there exists an open set of $\alpha\in\R$ such that \eqref{eq:enough} holds, which completes the proof.
\end{proof}

\begin{proof}[Proof of Proposition \ref{CH3_prop_dimNew}]
	Let $j\in J$ be arbitrary but fixed. For every $\ibar\in(I\cup J\setminus\{j\})^*$, let us define a map $g_{\ibar}^{(\alpha)}\colon\R\mapsto\R$ such that 
	\begin{equation}\label{eq:gdef}
	g_{\ibar}^{(\alpha)}(x):=\rho_j\wvec(c_j+\alpha\beta_j)^TA_{\ibar}\vvec_j\cdot x+\rho_j\wvec(c_j+\alpha\beta_j)^Tf_{\ibar}(\tvec_j).
	\end{equation}
	In particular, 
	\begin{equation}\label{eq:g}
		f_{j\ibar}(x\vvec_j+\underline{t}_j)=g_{\ibar}^{(\alpha)}(x)\vvec_j+\underline{t}_j\text{  for every $x\in\R$.}
	\end{equation}
	Furthermore, $\|A_jA_{\ibar}|\mathrm{Im}(A_j)\|=\left|\rho_j\wvec(c_j+\alpha\beta_j)^TA_{\ibar}\vvec_j\right|$ for every $\ibar\in(I\cup J\setminus\{j\})^*$.
		
	For every $n\in\N$, let us define an IFS on the real line as $\mathcal{G}^{n}_{\alpha}=\{g_{\ibar}^{(\alpha)}\}_{\ibar\in\bigcup_{k=0}^n(I\cup J\setminus\{j\})^k}$. Let us denote the attractor of $\mathcal{G}^{n}_{\alpha}$ by $\Gamma_{n,\alpha}$. Let $\pi_\alpha\colon\left((\bigcup_{k=0}^nI\cup J\setminus\{j\})^k\right)^\N\mapsto\Gamma_{n,\alpha}$ be the natural projection associated to the IFS $\mathcal{G}^{n}_{\alpha}$. Then by \eqref{eq:g},
	\begin{equation}\label{eq:conj}
	\Pi_\alpha(j\ibar_1j\ibar_2j\cdots)=\pi_\alpha(\ibar_1\ibar_2\cdots)\vvec_j+\tvec_j.
	\end{equation}
	By defining the bi-Lipschitz mapping $h\colon\R\mapsto\R^2$ as $h(x)=x\vvec_j+\tvec_j$, we see that $h(\Gamma_{n,\alpha})\subset\Lambda_\alpha$ for every $n\in\N$ and $\alpha\in\R$. Moreover, combining Lemma~\ref{lemma_for_Hochman} with \eqref{eq:conj}, we see that 
	$$
	\pi_\alpha(\ibar_1\ibar_2\cdots)\equiv\pi_\alpha(\jbar_1\jbar_2\cdots)\text{ if and only if }\ibar_k=\jbar_k\text{ for every }k=1,2,\ldots.
	$$
	
	Clearly, the contraction ratios and the translation parameters of the maps in $\mathcal{G}^{n}_{\alpha}$ are analytic maps of $\alpha\in\R$. Let $R_n$ be the set of roots of the contraction ratios of the maps in $\mathcal{G}^{n}_{\alpha}$. Then for every $a\in\N$, the set $(-a,a)\cap R_N$ is finite. Let $I_1^a,\ldots,I_N^a$ be disjoint open subintervals of $(-a,a)$ such that $\|A_jA_{\ibar}|\mathrm{Im}(A_j)\|\neq0$ for every $\alpha\in I_k$ and every $k=1,\ldots,L$. For every $k=1,\ldots,L$ and $\ell\in\N$, let $J_{k,\ell}^a\subseteq I_k^a$ be compact intervals such that $\bigcup_{\ell=1}^\infty J_{k,\ell}^a=I_k^a$.
	
	Applying Theorem~\ref{CH2_thm_Hoch2}, there exists $E_{k,\ell}^a\subset J_{k,\ell}^a$ such that $\dim_H(E_{k,\ell}^a)=0$ and
	$$
	\dim_H(\Gamma_{n,\alpha})=s^{(n)}(\alpha)\text{ for every $\alpha\in J_{k,\ell}^a\setminus E_{k,\ell}^a$},
	$$
	where $s^{(n)}(\alpha)$ is the similarity dimension of $\mathcal{G}^{n}_{\alpha}$, that is,
	$$
	\smashoperator{\sum_{\ibar\in\bigcup_{k=0}^n(I\cup J\setminus\{j\})^k}}\|A_jA_{\ibar}|\mathrm{Im}(A_j)\|^{s^{(n)}(\alpha)}=1.
	$$
	By Lemma~\ref{lem:affdim} and \eqref{eq:sj}, $\lim_{n\to\infty}s^{(n)}(\alpha)=s(\mathcal{F}_{\alpha})$. Hence, by choosing $E:=\bigcup_{n=1}^\infty R_n\cup\bigcup_{a,k,\ell}E_{k,\ell}^a$, the claim of the proposition follows.
\end{proof}

\section{Exceptional parameters}

The remaining part of the paper is devoted to prove Theorem~\ref{thm:main}\eqref{it:exception}. Let $\mathcal{F}_{\mathsf{w}}$ be a family of affine IFSs as in \eqref{eq:IFSdef} with attractor $\Lambda_{\mathsf{w}}$. Suppose that $\sup_{\mathsf{w}}s(\mathcal{F}_{\mathsf{w}})<1$ and $\mathcal{F}_{\mathsf{w}}$ satisfies the convex separation condition uniformly. 

\begin{lemma}\label{lem:commonfix}
	Let us fix $j\in J$ and $i\in I\cup J$ such that $i\neq j$. Then there exists $\wvec_j\in\mathbb{S}^1$ such that
	$f_{j}$ and $f_j\circ f_i$ share the same fixed point. In particular, $f_{j}\circ f_j\circ f_i\equiv f_j\circ f_i\circ f_{j}$.
\end{lemma}

\begin{proof}
	Let $U\subset\R^2$ be the convex set with respect to the convex separation condition holds uniformly. Let the map $g_{\emptyset}^{(\alpha)}\colon\R\mapsto\R$ and $g_{i}^{(\alpha)}\colon\R\mapsto\R$ be as in \eqref{eq:gdef}. Namely,
	\[
	\begin{split}
	g_{\emptyset}^{(\alpha)}(x)&=\rho_j\langle\wvec(\alpha),\vvec_j\rangle\cdot x+\rho_j\langle\wvec(\alpha),\tvec_j\rangle\text{ and}\\
	g_{i}^{(\alpha)}(x)&=\rho_j\langle\wvec(\alpha),A_i\vvec_j\rangle\cdot x+\rho_j\langle\wvec(\alpha),A_i\tvec_j+\tvec_i\rangle.
	\end{split}
\]
So, $f_j(x\vvec_j+\tvec_j)=g_{\emptyset}^{(\alpha)}(x)\vvec_j+\tvec_j$ and $f_j\circ f_i(x\vvec_j+\tvec_j)=g_{i}^{(\alpha)}(x)\vvec_j+\tvec_j$.	
	By Lemma~\ref{CH8_convex}, there exists $\alpha'\in[0,\pi]$ such that $\mathrm{proj}_{\underline{w}(\alpha')}(f_j(U))\cap\mathrm{proj}_{\underline{w}(\alpha')}(f_i(U))=\emptyset$, and similarly for $\underline{w}(\alpha'+\pi)$, where $\underline{w}(\alpha)=(\cos(\alpha),\sin(\alpha))^T$. This implies that 
	$$
	g_{\emptyset}^{(\alpha')}(x)<g_{i}^{(\alpha')}(y)\text{ and }g_{\emptyset}^{(\alpha'+\pi)}(x)>g_{i}^{(\alpha'+\pi)}(y)\text{ for all $x,y\in\R$ with }x\vvec_j+\tvec_j,y\vvec_j+\tvec_j\in f_j(U).
	$$
	Since the fixed point of the maps $g_{\emptyset}^{(\alpha)}$ and $g_{i}^{(\alpha')}$ are continuous functions of $\alpha$, by Bolzano-Darboux Theorem, there exist $\alpha\in[0,2\pi]$ such that $g_{\emptyset}^{(\alpha)}$ and $g_{i}^{(\alpha')}$ share the same fixed points, and in particular $f_j$ and $f_j\circ f_i$ do.
\end{proof}

\begin{proof}[Proof of Theorem~\ref{thm:main}\eqref{it:exception}]
	Let $\mathcal{F}_{\mathsf{w}}$ be a family of affine IFSs as in \eqref{eq:IFSdef} with attractor $\Lambda_{\mathsf{w}}$ such that $\#J\geq1$ and $\#(I\cup J)\geq2$. Suppose that $\sup_{\mathsf{w}}s(\mathcal{F}_{\mathsf{w}})<1$, $\mathcal{F}_{\mathsf{w}}$ satisfies the convex separation condition uniformly for every $\mathsf{w}\in\mathbb{T}^{\#J}$. 
	
	Let $j\in J$ and $i\in I\cup J$ be arbitrary but fixed such that $i\neq j$. By Lemma~\ref{lem:commonfix}, there exists $\wvec_j\in\mathbb{S}^1$ such that $f_{j}\circ f_j\circ f_i\equiv f_j\circ f_i\circ f_j$. Let us fix this  $\wvec_j\in\mathbb{S}^1$ and choose every other $\wvec_{j'}$ for $j'\in J\setminus\{j\}$ arbitrarily. Let us define a new IFS 
	$$
	\mathcal{F}_{\mathsf{w}}'=\{f_{\ibar}\}_{\ibar\in(I\cup J)^3}\setminus\{f_j\circ f_j\circ f_i\}.
	$$
	Hence, $\Lambda_\mathsf{w}'=\Lambda_\mathsf{w}$, where $\Lambda_\mathsf{w}'$ is the attractor of $\mathcal{F}_{\mathsf{w}}'$. 
	
	However, by Lemma~\ref{lem:affdim} and Lemma~\ref{lem:affdim2}, one can see that $s(\mathcal{F}_{\mathsf{w}}')<s(\mathcal{F}_{\mathsf{w}})$, and by K\"aenm\"aki and Nissinen \cite[Lemma~3.2]{kaenmaki2022non}, we have that
	$$
	\overline{\dim}_B(\Lambda_{\mathsf{w}})=\overline{\dim}_B(\Lambda_{\mathsf{w}}')\leq s(\mathcal{F}_{\mathsf{w}}')<s(\mathcal{F}_{\mathsf{w}}),
	$$
	which implies the desired claim.
\end{proof}

\bibliographystyle{abbrv}
\bibliography{Bibliography}

\end{document}